\documentclass[12pt,twoside]{amsart}
\date{10 September 2010}
\usepackage{latexsym,amsthm,amsmath,amsfonts,amscd,amssymb}
\usepackage{hyperref}
\usepackage[ansinew]{inputenc}
\usepackage[all]{xy}
\usepackage{graphics}
\usepackage{lscape}
\usepackage{array}
\theoremstyle{plain}  
\newtheorem{theorem}{Theorem}[section]

\newtheorem*{theorem*}{Theorem}

\newtheorem{corollary}[theorem]{Corollary}
\newtheorem{lemma}[theorem]{Lemma}
\newtheorem{proposition}[theorem]{Proposition}

\newtheorem{tech-lemma}[theorem]{Technical Lemma}
\newtheorem{definition}[theorem]{Definition}
\theoremstyle{remark}
\newtheorem{example}[theorem]{Example}

\newtheorem{remark}[theorem]{Remark}
\newtheorem*{remark*}{Remark}

\newtheorem*{claim*}{Claim}
\numberwithin{equation}{section}
\renewcommand{\leq}{\leqslant}

\renewcommand{\geq}{\geqslant}

\newcommand{\R}{\mathbb{R}}
\newcommand{\Z}{\mathbb{Z}}
\newcommand{\CC}{\mathbb{C}}
\newcommand{\HH}{\mathbb{H}}
\newcommand{\M}{{\mathcal M}}
\newcommand{\N}{{\mathcal N}}

\newcommand{\U}{\mathrm{U}}
\newcommand{\GL}{\mathrm{GL}}

\newcommand{\SL}{\mathrm{SL}}

\newcommand{\Sp}{\mathrm{Sp}}

\newcommand{\la}{\langle}
\newcommand{\ra}{\rangle}
\newcommand{\st}{\;|\;}

\DeclareMathOperator{\ad}{ad}
\DeclareMathOperator{\Ad}{Ad}
\DeclareMathOperator{\Aut}{Aut}

\DeclareMathOperator{\rk}{rk}
\DeclareMathOperator{\im}{im}
\DeclareMathOperator{\coker}{coker}
\DeclareMathOperator{\Hom}{Hom}
\DeclareMathOperator{\End}{End}

\newcommand{\liem}{\mathfrak{m}}
\newcommand{\liez}{\mathfrak{z}}

\newcommand{\liemc}{\mathfrak{m}^{\mathbb{C}}}
\newcommand{\lieh}{\mathfrak{h}}
\newcommand{\liehc}{\mathfrak{h}^{\mathbb{C}}}
\newcommand{\lieg}{\mathfrak{g}}
\newcommand{\liegc}{\mathfrak{g}^{\mathbb{C}}}

\let\oldmarginpar\marginpar
\renewcommand\marginpar[1]{\oldmarginpar{\tiny\bf\begin{flushleft} #1
\end{flushleft}}}

\begin{document}

%
%

\title[Higgs bundles for the non-compact dual of the unitary group]
{Higgs bundles for the non-compact dual of the unitary group}
%
%

\author[O. García-Prada]{Oscar García-Prada}
\address{O. García-Prada\\ Instituto de Ciencias Matemáticas\\
CSIC-UAM-UC3M-UCM\\
Serrano 121\\
28006 Madrid\\ Spain.}
\email{oscar.garcia-prada@uam.es}

\author[A. G. Oliveira]{André G. Oliveira}
\address{A. G. Oliveira\\ Departamento de Matemática\\
  Universidade de Trás-os-Montes e Alto Douro\\
  Quinta dos Prados\\ 5000-911 Vila Real \\ Portugal.}
\email{agoliv@utad.pt}

\thanks{%
Authors partially supported by Ministerio de
  Educación y Ciencia and Conselho de Reitores das
  Universidades Portuguesas through Acção Integrada Luso-Espanhola nº E-38/09, Moduli of quadratic bundles (2009-2010).
First author partially supported by the Spanish Ministerio de Ciencia e Innovaci\'on
(MICINN) under grant MTM2007-67623. The first author thanks
the Max Planck Institute for Mathematics in Bonn --- that he was
visiting when this paper
was completed --- for support.
  Second author partially supported by Fundação para a Ciência e Tecnologia - FCT (Portugal) through the Centro de Matemática da Universidade de Trás-os-Montes e Alto Douro - CMUTAD, and through Project PTDC/MAT/099275/2008.
Both authors thank the referee for invaluable suggestions that improved the presentation and for drawing our attention to a gap in the first version of the paper.}

\begin{abstract}
Using Morse-theoretic techniques, we show that the moduli space of $\U^*(2n)$-Higgs bundles over a compact Riemann surface is connected.
\end{abstract}

\maketitle

%
%

\section{Introduction}\label{section:Introduction}

Let $X$ be a compact Riemann surface of genus $g\geq 2$, and let
$\M_G$ be the moduli space of polystable $G$-Higgs bundles over $X$, 
where $G$ is a real reductive Lie group. $G$-Higgs bundles for  $G=\GL(m,\CC)$ were 
introduced by Hitchin in \cite{hitchin:1987}. In this case a  $G$-Higgs
bundle is a pair $(V,\varphi)$ consisting of a holomorphic bundle $V$ over $X$ 
and a holomorphic section $\varphi$ of the bundle $\End V$ twisted with the
canonical bundle of $X$.
In this paper we study the moduli space $\M_{\U^*(2n)}$,
where $\U^*(2n)$ is the subgroup of $\GL(2n,\CC)$ consisting of matrices $M$ 
verifying that $\overline  MJ_n=J_nM$ where
$J_n=\begin{pmatrix}
     0 & I_n \\
     -I_n & 0
\end{pmatrix}.$
This group is the non-compact dual of $\U(2n)$ in the sense that 
the non-compact  symmetric space $\U^*(2n)/\Sp(2n)$ is the dual of the 
compact symmetric space $\U(2n)/\Sp(2n)$ in Cartan's classification
of symmetric spaces (cf. \cite{helgason:2001}).

The complex general linear group $\GL(m,\CC)$ has as real forms the groups $\U(p,q)$, with $p+q=m$ 
(including the compact real form $\U(m)$), the split real form $\GL(m,\R)$ and, when $m=2n$  also $\U^*(2n)$. 
In a similar fashion as the images in $\M_{\GL(m,\CC)}$ of the moduli spaces $\M_{\U(p,q)}$ with $p+q=m$  appear as fixed point subvarieties of 
$\M_{\GL(m,\CC)}$ under the involution $(V,\varphi)\mapsto (V,-\varphi)$,  the image of the moduli space $\M_{\U^*(2n)}$ in $\M_{\GL(m,\CC)}$ is a fixed point subvariety of  $\M_{\GL(2n,\CC)}$ under the involution 
$(V,\varphi)\mapsto (V^*,\varphi^t)$, together with $\M_{\GL(2n,\R)}$ (see \cite{garcia:2007} and \cite{garcia-prada-ramanan}).
The number of connected components of $\M_{\U(p,q)}$ were determined in 
\cite{bradlow-garcia-prada-gothen:2003} and the ones of $\M_{\GL(m,\R)}$ in 
\cite{bradlow-garcia-prada-gothen:2004} and \cite{hitchin:1992}. The real form $\U^*(2n)$ was therefore the remaining one for which the number of connected components was still to be determined. 
In this paper we prove  the following.
\begin{theorem*}
The moduli space $\M_{\U^*(2n)}$ of $\U^*(2n)$-Higgs bundles over $X$ is connected.
\end{theorem*}

We adopt the Morse-theoretic approach pioneered by Hitchin in
\cite{hitchin:1987}, and which has already been applied for several other groups (see, for example, \cite{hitchin:1992,gothen:2001,bradlow-garcia-prada-gothen:2003,bradlow-garcia-prada-gothen:2004,garcia-prada-mundet:2004,bradlow-garcia-prada-gothen:2005,garcia-gothen-mundet:2008 II}), to reduce our problem to the study of connectedness of certain subvarieties of $\M_{\U^*(2n)}$. For that, we obtain first  a detailed description of smooth points of the moduli space $\M_{\U^*(2n)}$, then 
we give also an explicit description of stable and non-simple $\U^*(2n)$-Higgs bundles, 
and show how the polystable $\U^*(2n)$-Higgs bundles split as a direct sum of stable objects.

Non-abelian Hodge theory on $X$ establishes a homeomorphism between $\M_G$ and the moduli space of reductive representations of $\pi_1X$ in $G$ (cf. \cite{hitchin:1987,simpson:1988,simpson:1992,donaldson:1987,corlette:1988,garcia-gothen-mundet:2008,bradlow-garcia-prada-mundet:2003}). A direct consequence of our result is thus the following.
\begin{theorem*}
The moduli space of reductive representations of $\pi_1X$ in $\U^*(2n)$ is connected.
\end{theorem*}

The connectedness of $\M_{\U^*(2n)}$ reflects the fact that $\U^*(2n)$ is simply-connected. 
It seems plausible that, like for $\M_{\U^*(2n)}$, $\M_G$ is connected whenever $G$ is a real reductive Lie group with $\pi_1G=0$. When $G$ is complex this has been proved by Hitchin \cite{hitchin:1987}
for $\SL(2,\CC)$ and by Simpson \cite{simpson:1988} for $\SL(n,\CC)$. For general complex $G$, the
result follows from a theorem by Li \cite{li:1993}, showing the analogous result for the moduli space of flat $G$-connections,  and the non-abelian Hodge theory correspondence.
As far as we know, there is no proof in general using Higgs bundle techniques.

\section{$\U^*(2n)$-Higgs bundles}\label{section:U*(2n)-Higgs bundles}

Let $X$ be a compact Riemann of genus $g\geq 2$, and let $G$ a real reductive Lie group.
Let $H\subseteq G$ be a maximal compact subgroup
and $H^\CC\subseteq G^\CC$ their complexifications. Let
$$\lieg=\lieh\oplus\liem$$ be a Cartan
decomposition of $\lieg$,
where $\liem$ is the complement of $\lieh$ with
respect to the non-degenerate $\Ad(G)$-invariant bilinear form on $\lieg$. If
$\theta:\lieg\to\lieg$ is the
corresponding Cartan involution then $\lieh$ and
$\liem$ are its $+1$-eigenspace and $-1$-eigenspace,
respectively. Complexifying, we have the decomposition
$$\liegc=\liehc\oplus\liemc$$ and
$\liemc$ is a representation of $H^\CC$ through the
so-called \emph{isotropy representation}
\begin{equation}\label{complex isotropy representation}
 \iota:H^\CC\longrightarrow\Aut(\liemc)
\end{equation}
which is induced by the adjoint
representation of $G^\CC$ on $\liegc$.
If $E_{H^\CC}$ is a principal $H^{\CC}$-bundle over $X$, we denote by
$E_{H^\CC}(\liemc)=E_{H^\CC}\times_{H^{\CC}}\liemc$ the
vector bundle, with fiber $\liemc$, associated to the
isotropy representation.

Let $K=T^*X^{1,0}$ be the canonical line bundle of $X$.
\begin{definition}\label{definition of Higgs bundle}
Let $X$ be a compact Riemann surface of genus $g\geq 2$. A \emph{$G$-Higgs bundle} over $X$ is a pair
$(E_{H^\CC},\varphi)$ where $E_{H^\CC}$ is a principal holomorphic $H^\CC$-bundle over
$X$ and $\varphi$ is a global holomorphic section of
$E_{H^\CC}(\liemc)\otimes K$, called the \emph{Higgs field}.
\end{definition}

Higgs bundles were introduced by Hitchin \cite{hitchin:1987} on compact Riemann surfaces and by Simpson \cite{simpson:1992} on any compact Kähler manifold.

\begin{example}\label{examples of G-Higgs bundles}\mbox{}
\begin{enumerate}
\item If $G$ is compact, a $G$-Higgs bundle is simply a holomorphic $G^\CC$-principal bundle. For instance, a $\U(n)$-Higgs bundle is simply a holomorphic $\GL(n,\CC)$-principal bundle over $X$ or, in terms of holomorphic vector bundles, a $\U(n)$-Higgs bundle is a rank $n$ holomorphic vector bundle.
\item  If $G$ is complex with maximal compact $H$, we have that $H^\CC=G$ and $\liem=\sqrt{-1}\lieh$, so $\liemc=\lieg$. Thus a $G$-Higgs bundle is 
a pair $(E_G,\varphi)$ where $E_G$ is a holomorphic $G$-bundle and $\varphi\in H^0(E_G(\lieg)\otimes K)$ where $E_G(\lieg)$ denotes the adjoint bundle of $E_G$, obtained from $E_G$ under the adjoint action of $G$ on $\lieg$: $E_G(\lieg)=E_G\times_{\Ad}\lieg$. As an example, a $\GL(m,\CC)$-Higgs bundle is, in terms of vector bundles, a pair $(V,\varphi)$ with $V$ a holomorphic rank $m$ vector bundle and $\varphi\in H^0(\End(V)\otimes K)$.
\end{enumerate}
\end{example}

Let us now consider the case of the real Lie group $\U^*(2n)$. A possible way to realize the group $\U^*(2n)$ as a matrix group is as the subgroup of $\GL(2n,\CC)$ defined as
$$\U^*(2n)=\left\{M\in\GL(2n,\CC)\st\overline  MJ_n=J_nM\right\},$$ where
$$J_n=\begin{pmatrix}
     0 & I_n \\
     -I_n & 0
\end{pmatrix}.$$
From this definition, it is obvious that $\U^*(2n)$ is the real form of $\GL(2n,\CC)$ given by the fixed point set of the involution $\sigma:\GL(2n,\CC)\to\GL(2n,\CC)$, $\sigma(M)=J_n^{-1}\overline{M}J_n$.

\begin{remark}
$\U^*(2n)$ is also the group of quaternionic linear automorphisms of an $n$-dimensional vector space over the ring $\mathbb{H}$ of quaternions, and therefore $\U^*(2n)$ is also denoted by $\GL(n,\mathbb{H})$.
\end{remark}

A maximal compact subgroup of $\U^*(2n)$ is the compact symplectic group $\Sp(2n)$ (or, equivalently, the group of $n\times n$ quaternionic unitary matrices), whose complexification is $\Sp(2n,\CC)$, the complex symplectic group.

The corresponding Cartan decomposition of the complex Lie algebras is 
$$\mathfrak{gl}(2n,\CC)=\mathfrak{sp}(2n,\CC)\oplus\liemc,$$
where $\liemc=\{A\in\mathfrak{gl}(2n,\CC)\st A^tJ_n=J_nA\}$. Hence:

\begin{definition}
A $\U^*(2n)$-Higgs bundle over $X$ is a pair $(E,\varphi)$, where $E$ is a holomorphic $\Sp(2n,\CC)$-principal bundle and the Higgs field $\varphi$ is a global holomorphic section of $E\times_{\Sp(2n,\CC)}\liemc\otimes K$.
\end{definition}
 Now, if $\mathbb{W}$ is the standard $2n$-dimensional complex representation of $\Sp(2n,\CC)$ and $\Omega$ denotes the standard symplectic form on $\mathbb{W}$, then the isotropy representation space is $$\liemc=S^2_{\Omega}\mathbb{W}=\{\xi\in\End(\mathbb{W})\,|\,\Omega(\xi\cdot,\cdot)=\Omega(\cdot,\xi\cdot)\}\subset\End(\mathbb{W}).$$ 

Given a symplectic vector bundle $(W,\Omega)$, denote by $S^{2}_{\Omega}W$ the bundle of endomorphisms $\xi$ of $W$ which are symmetric with respect to $\Omega$ i.e. such that $\Omega(\xi\,\cdot,\cdot)=\Omega(\cdot,\xi\,\cdot)$. In the vector bundle language, we have hence the following:
\begin{definition}
A \emph{$\U^*(2n)$-Higgs bundle} over $X$ is a triple $(W,\Omega,\varphi)$,
where $W$ is a holomorphic vector bundle of rank $2n$, $\Omega\in H^0(\Lambda^2W^*)$ is a symplectic form on $W$ and the Higgs field
$\varphi\in H^0(S_{\Omega}^2 W\otimes K)$ is a $K$-twisted endomorphism $W\to W\otimes K$, symmetric with respect to $\Omega$. 
\end{definition}

Given the symplectic form $\Omega$, we have the usual skew-symmetric isomorphism $$\omega:W\stackrel{\cong}{\longrightarrow}W^*$$ given by $$\omega(v)=\Omega(v,-).$$ It follows from the symmetry of $\varphi$ with respect to $\Omega$ that
\begin{equation}\label{symmetric higgs field}
(\varphi^t\otimes 1_K)\omega=(\omega\otimes 1_K)\varphi.
\end{equation}

\begin{remark}\label{equivalent definition of U*(2n)}
Given a $\U^*(2n)$-Higgs bundle $(W,\Omega,\varphi)$, define the homomorphism $$\tilde\varphi:W^*\longrightarrow W\otimes K$$  by
\begin{equation}\label{tildevarphi}
\tilde\varphi=\varphi\omega^{-1}.
\end{equation}
It follows from (\ref{symmetric higgs field}) that $\tilde\varphi$ is skew-symmetric i.e.
$$\tilde\varphi^t\otimes 1_K=-\tilde\varphi.$$ In other words, $$\tilde\varphi\in H^0(\Lambda^{2}W\otimes K).$$
Hence, since $\omega:W\to W^*$ is an isomorphism, \emph{it is equivalent to think of a $\U^*(2n)$-Higgs bundle as a triple $(W,\Omega,\varphi)$ with $\varphi\in H^0(S_{\Omega}^2W\otimes K)$ or as a triple $(W,\Omega,\tilde\varphi)$ with $\tilde\varphi\in H^0(\Lambda^2W\otimes K)$}.
\end{remark}

Given a $\U^*(2n)$-Higgs bundle $(W,\Omega,\varphi)$, we must then have $W\cong W^*$, thus $$\deg(W)=0.$$ In other words, the topological invariant of these objects given by the degree is always zero. This is of course consequence of the fact that the group $\U^*(2n)$ is connected and simply-connected and that, for $G$ connected, $G$-Higgs bundles are topologically classified (cf. \cite{ramanathan:1975}) by the elements of $\pi_1G$.

\begin{remark}
 Two $G$-Higgs bundles $(E_{H^\CC},\varphi)$ and $(E_{H^\CC}',\varphi')$ over $X$ are \emph{isomorphic} if there is an holomorphic isomorphism $f:E_{H^\CC}\to E_{H^\CC}'$ such that $\varphi'=\tilde f(\varphi)$, where $\tilde f\otimes 1_K:E_{H^\CC}(\liemc)\otimes K\to E_{H^\CC}'(\liemc)\otimes K$ is the map induced from $f$ and from the isotropy representation $H^\CC\to\Aut(\liemc)$.
Hence, two $\U^*(2n)$-Higgs bundles $(W,\Omega,\varphi)$ and
$(W',\Omega',\varphi')$ are \emph{isomorphic} if there is an isomorphism $f:W\to W'$ 
such that $\omega=f^t\omega'f$ and $\varphi' f=(f\otimes 1_K)\varphi$, which is equivalent to $\tilde\varphi'=(f\otimes 1_K)\tilde\varphi f^t$ where $\tilde\varphi$ is given in (\ref{tildevarphi}).
\end{remark}

\section{Moduli spaces}
\subsection{Stability conditions}
Now we consider the moduli space of $G$-Higgs bundles, for which we need the notions of stability, semistability and polystability.

In order to find these notions for $\U^*(2n)$-Higgs bundles, we briefly recall here the main
definitions. The main reference is \cite{garcia-gothen-mundet:2008}, where the general notion of (semi,poly)stability is deduced in detail and where several examples are studied. Let $E_{H^\CC}$ be a principal $H^\CC$-bundle. Let $\Delta$ be  a
fundamental system of roots of the Lie algebra $\lieh^\CC$. For every subset
$A\subseteq \Delta$
there is a corresponding parabolic subgroup $P_A\subseteq H^\CC$. 
Let $\chi:\mathfrak{p}_A\to\CC$ be an antidominant character of $\mathfrak{p}_A$, the Lie algebra of $P_A$.
Let $\sigma$ be a holomorphic section of $E_{H^\CC}(G/P_A)$, that is, a
reduction of the structure group of $E_{H^\CC}$ to $P_A$, and denote by
$E_{\sigma}$ the corresponding $P_A$-bundle. We define the \emph{degree} of
$E_{H^\CC}$ with respect to $\sigma$ and $\chi$ by
$$\deg E_{H^\CC}(\sigma,\chi)=\deg\chi_{*}E_{\sigma}.$$
When $\chi$ lifts to a character $P_A\to\CC^*$, then the degree of $E_{H^\CC}$ is written in terms of the degree of the line bundle obtained from $E_{\sigma}$ and from the character $P_A\to\CC^*$. When $\chi$ does not lift, the degree of $E_{H^\CC}$ is also the degree of a certain line bundle obtained also from $E_{\sigma}$ and $\chi$. There is also an alternative definition of degree, in terms of Chern-Weil theory. The detailed definitions of degree can be found in Sections 2.3-2.6 of \cite{garcia-gothen-mundet:2008}.

Let $\iota:H^\CC\rightarrow \Aut(\liem^\CC)$ be
the isotropy representation. We define
\begin{align*}
&\liem_{\chi}^-=\{v\in \liem^\CC\ :\ \iota(e^{ts_{\chi}})v
\text{ is bounded as } t\to\infty\}\\
&\liem^0_{\chi}=\{v\in \liem^\CC\ :\ \iota(e^{ts_{\chi}})v=v\;\;
\mbox{for every} \;\; t\}.
\end{align*}

Let $\la\cdot,\cdot\ra$ be an invariant $\CC$-bilinear pairing on
$\liehc$.

Here is the general definition of semistability, given in \cite{garcia-gothen-mundet:2008}. It depends on a parameter $\alpha\in \sqrt{-1}\lieh\cap\liez$, where $\liez$ is the center of $\liehc$.
\begin{definition}\label{semistability definition}
Let $\alpha\in\sqrt{-1}\lieh\cap\liez$. A $G$-Higgs bundle $(E_{H^\CC},\varphi)$ is:
\begin{itemize}
 \item \emph{$\alpha$-semistable}
if $$\deg E_{H^\CC}(\sigma,\chi)-\langle\alpha,\chi\rangle\geq 0,$$ for any parabolic subgroup $P_A$ of $H^\CC$, any antidominant
character $\chi$ of $\mathfrak{p}_A$
and any reduction of structure group $\sigma$ of $E_{H^\CC}$ to $P_A$
such that $\varphi\in H^0(E_{\sigma}(\liem_{\chi}^-)\otimes K)$, where $E_{\sigma}$ is the corresponding $P_{A}$-bundle.
\item \emph{$\alpha$-stable} if it is $\alpha$-semistable, and $$\deg
E_{H^\CC}(\sigma,\chi)-\langle\alpha,\chi\rangle>0,$$ for any $P_A$,
$\chi$ and $\sigma$ as above such that $\varphi\in
H^0(E_{\sigma}(\liem_{\chi}^-)\otimes K)$, $A\not=\emptyset$ and $\chi\not\in(\lieh\cap\liez)^*$.
\item \emph{$\alpha$-polystable} if it is $\alpha$-semistable, and for each
$P_A$, $\sigma$ and $\chi$ as above such that $$\deg E_{H^\CC}(\sigma,\chi)-\langle\alpha,\chi\rangle=0,$$ there exists a
holomorphic reduction of the structure group,
$\sigma_L\in H^0(E_{\sigma}(P_A/L_A))$, of $E_{\sigma}$ to the
Levi subgroup $L_A$ of $P_A$, such that $\varphi\in H^0(E_{\sigma_L}(\liem^0_{\chi})\otimes K)$, where $E_{\sigma_L}$ is the corresponding $L_A$-bundle.
\end{itemize}
\end{definition}

\subsubsection{$\GL(n,\CC)$-Higgs bundles}
Recall from Example \ref{examples of G-Higgs bundles} that a $\GL(n,\CC)$-Higgs bundle is a pair $(V,\varphi)$ where $V$ is a rank $n$ vector bundle and $\varphi\in H^0(\End(V)\otimes K)$.
A subbundle $V'$ of $V$ is said to be \emph{$\varphi$-invariant} if $\varphi(V')\subset V'\otimes K$.

The \emph{slope} of a vector bundle $V$ is defined by the quotient $\mu(V)=\deg(V)/\rk(V)$, where $\rk(V)$ denotes the rank of $V$.

It can be seen that, when applied to $\GL(n,\CC)$-Higgs bundles, the (semi,poly)stability condition of Definition \ref{semistability definition}, simplifies as follows:
\begin{proposition}\label{(poly)stability for usual Higgs bundles}
A $\GL(n,\CC)$-Higgs bundle $(V,\varphi)$ is:
\begin{itemize}
\item \emph{semistable} if and only if $\mu(V')\leq\mu(V)$ for every proper $\varphi$-invariant subbundle $V'\subsetneq V$.
\item \emph{stable} if and only if $\mu(V')<\mu(V)$ for every proper $\varphi$-invariant subbundle $V'\subsetneq V$.
\item \emph{polystable} if and only if it is semistable and, for every proper $\varphi$-invariant subbundle $V'\subsetneq V$ such that $\mu(V')=\mu(V)$, there is another proper $\varphi$-invariant subbundle $V''\subsetneq V$ such that $\mu(V'')=\mu(V)$ and $V=V'\oplus V''$.
\end{itemize}
\end{proposition}

Notice that, on the last item of the previous proposition, $(V',\varphi|_{V'})$ and $(V'',\varphi|_{V''})$ must also be polystable (this holds due to the Hitchin-Kobayashi correspondence between polystable $\GL(n,\CC)$-Higgs bundles and solutions to the $\GL(n,\CC)$-Hitchin equations; see \cite{hitchin:1987}). So, an iteration procedure shows that a $\GL(n,\CC)$-Higgs bundle $(V,\varphi)$ is polystable if and only if $V=V_1\oplus\cdots\oplus V_k$, where $\varphi(V_i)\subset V_i\otimes K$ and $(V_i,\varphi|_{V_i})$ are stable $\GL(\rk(V_i),\CC)$-Higgs bundles with $\mu(V_i)=\mu(V)$ (cf. \cite{nitsure:1991}).

\begin{remark}
It can be seen that a $\GL(n,\CC)$-Higgs bundle $(V,\varphi)$ is $\alpha$-semistable if and only if $\alpha=\mu(V)$ and $\mu(V')\leq\mu(V)$ for all proper subbundle $V'\subsetneq V$ such that $\varphi(V')\subset V'\otimes K$. So, the parameter is  fixed by the topological type of the $\GL(n,\CC)$-Higgs bundle. Hence, $\alpha=\mu(V)$ is the value of the parameter which we are considering in the previous proposition.
\end{remark}

\subsubsection{$\U^*(2n)$-Higgs bundles}
Also the general definition of (semi,poly)stability for
$G$-Higgs bundles given above, simplifies in the case $G=\U^*(2n)$, as we shall now briefly explain. The main reference for this, and where this is done in detail for several groups, is again \cite{garcia-gothen-mundet:2008}. In order to state the stability condition for $\U^*(2n)$-Higgs bundles, we first introduce some notation.

For any filtration of vector bundles
$$\mathcal W=(0=W_0\subsetneq W_1\subsetneq W_2\subsetneq\dots\subsetneq W_k=W)$$
satisfying $W_{k-j} = W_j^{\perp_\Omega}$ (here $W_j^{\perp_\Omega}$
denotes the orthogonal complement of $W_j$ with respect to
$\Omega$), define
$$\Lambda(\mathcal W)=\{(\lambda_1,\lambda_2,\dots,\lambda_k)\in\R^k
\mid \lambda_i\leq \lambda_{i+1}\text{ and
}\lambda_{k-i+1}=-\lambda_i\text{ for any }i\}.$$ For each $\lambda=(\lambda_1,\lambda_2,\dots,\lambda_k)\in\Lambda(\mathcal W)$ consider the subbundle of $S^2_{\Omega}W\otimes K$ given by
\begin{equation}\label{eq:N(calW,lambda)}
N(\mathcal W,\lambda)=S^2_{\Omega}W\otimes K\cap\sum_{\lambda_j\leq\lambda_i} \Hom(W_i,W_j)\otimes K\subseteq S^2_{\Omega}W\otimes K
\end{equation}
and let
\begin{equation}\label{eq:d(calW,lambda)}
d(\mathcal W,\lambda) = \sum_{j=1}^{k-1}(\lambda_j-\lambda_{j+1})\deg W_j.
\end{equation}

According to \cite{garcia-gothen-mundet:2008} the stability conditions for a $\U^*(2n)$-Higgs bundle can now be
stated as follows.
\begin{proposition}\label{prop:u*2n-alpha-stability}
A $\U^*(2n)$-Higgs bundle
$(W,\Omega,\varphi)$ is:
\begin{itemize}
 \item\emph{semistable} if and only if $d(\mathcal W,\lambda)\geq 0$ for
  every filtration $\mathcal W$ as above and any
  $\lambda\in\Lambda(\mathcal W)$ such that
  $\varphi \in H^0(N(\mathcal W,\lambda))$.
\item\emph{stable} if and only if it is semistable and $d(\mathcal W,\lambda)>0$ for every choice of
  filtration $\mathcal W$ and any nonzero $\lambda\in\Lambda(\mathcal W)$ such that
  $\varphi \in H^0(N(\mathcal W,\lambda))$.
\item\emph{polystable} if and only if it is
 semistable and, for every filtration $\mathcal W$ as above and any
 $\lambda\in\Lambda(\mathcal W)$ satisfying $\lambda_i<\lambda_{i+1}$
 for each $i$, such that $\varphi\in H^0(N(\mathcal W,\lambda))$ and
 $d(\mathcal W,\lambda)=0$, there is an isomorphism
 $$W\simeq W_1\oplus W_2/W_1\oplus\dots\oplus W_k/W_{k-1}$$
 such that $$\Omega(W_i/W_{i-1},W_j/W_{j-1})=0,\ \text{ unless }\ i+j=k+1.$$
 Furthermore, via this isomorphism,
 $$\varphi\in H^0\bigg(\bigoplus_i\End(W_i/W_{i-1})\otimes K\bigg).$$
\end{itemize}
\end{proposition}

\begin{remark}
The center of $\mathfrak{sp}(2n,\CC)$ is trivial hence, in the case of $G=\U^*(2n)$, the only possibility for the value of the parameter $\alpha$ of Definition \ref{semistability definition} is $\alpha=0$. So this is the value of $\alpha$ that we are considering in the previous proposition.
\end{remark}

There is a simplification of the stability condition for $\U^*(2n)$-Higgs bundles analogous to the cases considered in \cite{garcia-gothen-mundet:2008}. Recall that a subbundle $W'$ of $W$ is $\varphi$-invariant if $\varphi(W')\subset W'\otimes K$ i.e. $\varphi|_{W'}\in H^0(\End(W')\otimes K)$.

\begin{proposition}\label{prop:orthogonal-stability}
  A $\U^*(2n)$-Higgs bundle $(W,\Omega,\varphi)$
  is:
\begin{itemize}
 \item semistable if and only if $  \deg W'\leq 0$ for any isotropic and $\varphi$-invariant subbundle $W'\subset W$.
 \item stable if and only if it is semistable and $\deg W'<0$ for any isotropic and $\varphi$-invariant strict subbundle $0\neq W' \subset W$.
 \item polystable if and only if it is semistable and, for any isotropic (resp. coisotropic) and $\varphi$-invariant strict
  subbundle $0\neq W'\subset W$ such that $\deg W'=0$, there is another coisotropic (resp. isotropic) and $\varphi$-invariant subbundle $0\neq W''\subset W$ such that $W\cong W'\oplus W''$.
\end{itemize}
\end{proposition}

\begin{proof}
The proof follows \textit{mutatis mutandis} the proof of Theorem 4.4 of \cite{garcia-gothen-mundet:2008}. Take a
$\U^*(2n)$-Higgs bundle $(W,\Omega,\varphi)$, and assume
that $\deg W'\leq 0$ for any isotropic, $\varphi$-invariant, subbundle $W'\subset W$. We want to prove that $(W,\Omega,\varphi)$
is semistable. Suppose that $\varphi$ is nonzero, for otherwise the result follows from the usual characterization
 of (semi)stability for $\Sp(2n,\CC)$-principal bundles due to Ramanathan (see Remark 3.1 of \cite{ramanathan:1975}).

Choose any filtration $$\mathcal W=(0\subsetneq W_1\subsetneq W_2\subsetneq\dots\subsetneq W_k=W)$$ satisfying
$W_{k-j}=W_j^{\perp_{\Omega}}$ for any $j$, and consider the convex
set
\begin{equation}\label{eq:convex set}
\Lambda(\mathcal W,\varphi)= \{\lambda\in\Lambda(\mathcal W)\mid \varphi\in H^0(N(\mathcal W,\lambda))\}\subset\R^k,
\end{equation}
where $N(\mathcal W,\lambda)$ is defined in (\ref{eq:N(calW,lambda)}).

Let \begin{equation}\label{eq:JJJ}
\mathcal J=\{j\mid
\varphi(W_j)\subset W_j\otimes K\}= \{j_1,\dots,j_r\}.
\end{equation}
One checks easily that if $\lambda=(\lambda_1,\dots,\lambda_k)\in\Lambda(\mathcal W)$ then
\begin{equation}\label{eq:Lambda-JJJ}
\lambda\in\Lambda(\mathcal W,\varphi)
\Longleftrightarrow \lambda_a=\lambda_b\text{ for any }j_i<a\leq b\leq j_{i+1}.
\end{equation} We claim that the set of indices $\mathcal J$ is symmetric:
\begin{equation}
\label{eq:JJJ-symmetric} j\in\mathcal J \Longleftrightarrow k-j\in\mathcal J.
\end{equation}
Checking this is equivalent to prove that $\varphi(W_j)\subset W_j\otimes
K$ implies that $\varphi(W_j^{\perp_{\Omega}})\subset W_j^{\perp_{\Omega}}\otimes
K$. Suppose that this is not true, so that for
some $j$ we have $\varphi(W_j)\subset W_j\otimes K$ and there exists
some $w\in W_j^{\perp_{\Omega}}$ such that $\varphi(w)\notin W_j^{\perp_{\Omega}}\otimes
K$. Then there exists $v\in W_j$ such that
$\Omega(v,\varphi(w))\neq 0$. However, since $\varphi$ is symmetric with respect to $\Omega$, we must have $\Omega(v,\varphi(w))=\Omega(\varphi(v),w)$, and the latter vanishes because  by
assumption $\varphi(v)$ belongs to $W_j$. So we have reached a
contradiction, and (\ref{eq:JJJ-symmetric}) holds.

Let $\mathcal J'=\{j\in\mathcal J\mid 2j\leq k\}$ and, for each
$j\in\mathcal J'$, define the vector
$$L_j=-\sum_{c\leq j}e_c+\sum_{d\geq k-j+1}e_d$$
where $e_1,\dots,e_k$ is the canonical basis of $\R^k$. It
follows from (\ref{eq:Lambda-JJJ}) and (\ref{eq:JJJ-symmetric})
that $\Lambda(\mathcal W,\varphi)$ is the positive span of the vectors $\{
L_j\mid j\in\mathcal J'\}$. Hence,
\begin{equation}\label{eq:equivalent dgeq 0}
d(\mathcal W,\lambda)\geq 0\text{ for any }\lambda\in\Lambda(\mathcal W,\varphi)
\Longleftrightarrow d(\mathcal W,L_j)\geq 0\text{ for any }j\in\mathcal{J}'.
\end{equation}
Now, we compute $d(\mathcal W,L_j)=-\deg W_{k-j}-\deg W_j$. On the other
hand, since $W_{k-j}=W_j^{\perp_\Omega}$, we have an exact sequence $0\to W_{k-j}\to W\to
W_j^*\to 0$ (the projection is given by $v\mapsto\Omega(v,-)$), so $0=\deg W^*=\deg W_{k-j}+\deg W_j^*$,
hence $\deg W_{k-j}=\deg W_j$. Therefore $d(\mathcal W,L_j)\geq 0$ is
equivalent to $\deg W_j\leq 0$, which holds by assumption, because $W_j$ is $\varphi$-invariant and isotropic for every $j\in\mathcal J'$. Hence, from (\ref{eq:equivalent dgeq 0}) and Proposition \ref{prop:u*2n-alpha-stability}, it follows that $(W,\Omega,\varphi)$ is semistable.

The converse statement, namely, that if $(W,\Omega,\varphi)$ is
semistable then for any isotropic and $\varphi$-invariant subbundle $W'\subset W$ we have $\deg W'\leq 0$, is
immediate by applying the stability condition of Proposition \ref{prop:u*2n-alpha-stability} to the filtration
$0\subset W'\subset W'^{\perp_{\Omega}}\subset W$.

The proof of the second statement on stability is very similar to the  case of semistability, so we omit it.

Let us now consider the statement on polystability. Let $(W,\Omega,\varphi)$ be a semistable $\U^*(2n)$-Higgs bundle such that, for any isotropic and $\varphi$-invariant strict subbundle $ 0\neq W' \subset W$ such that  $\deg W' = 0$, there is another coisotropic and $\varphi$-invariant subbundle $0\neq W''\subset W$ such that $W=W'\oplus W''$. We want to prove that $(W,\Omega,\varphi)$ is polystable. Take any filtration $$\mathcal W=(0\subsetneq W_1\subsetneq W_2\subsetneq\dots\subsetneq W_k=W)$$ satisfying
$W_{k-j}=W_j^{\perp_{\Omega}}$ for any $j$, and the convex set $\Lambda(\mathcal W,\varphi)$ defined in (\ref{eq:convex set}). Let
\begin{equation}\label{eq:lambdainLambda}
\lambda\in\Lambda(\mathcal W,\varphi)
\end{equation}
 satisfying
\begin{equation}\label{eq:lambdaj<lambdaj+1}
\lambda_j<\lambda_{j+1}
\end{equation}
for every $j$, and such that
\begin{equation}\label{eq:d(W,lambda)=0}
d(\mathcal W,\lambda)=0.
\end{equation}
From (\ref{eq:lambdainLambda}), (\ref{eq:lambdaj<lambdaj+1}) and (\ref{eq:Lambda-JJJ}), we conclude that $$\mathcal J=\{1,\ldots,k\}$$ where $\mathcal J$ is given in (\ref{eq:JJJ}). Therefore, every $W_j$ in the filtration $\mathcal W$ is a $\varphi$-invariant subbundle of $W$.
Now, using the same arguments as in the proof of the semistability condition above, we conclude from (\ref{eq:d(W,lambda)=0}),
that $$\deg(W_i)=0,$$ for all $i\in\mathcal J'=\{1,\ldots,[k/2]\}$. Each of these $W_i$ is a strict isotropic and $\varphi$-invariant subbundle of $W$. In particular this holds for $W_1$, so from our assumption, we know that $W/W_1$ is a $\varphi$-invariant coisotropic subbundle of $W$ and
$W\cong W_1\oplus W/W_1$. The same is true for $W_i$ with $i\in\mathcal J'$, hence $$W\cong W_1\oplus W_2/W_1\oplus\dots\oplus W/W_{[k/2]}.$$ For $i\in\mathcal J\setminus\mathcal J'$, $W_i$ is a strict coisotropic and $\varphi$-invariant subbundle of $W$, so $W/W_i$ is a $\varphi$-invariant isotropic subbundle of $W$, and $W\cong W_i\oplus W/W_i$. Thus $$W\cong W_1\oplus W_2/W_1\oplus\dots\oplus W_{k-1}/W_{k-2}\oplus W/W_{k-1}.$$
Since $W_{k-j}=W_j^{\perp_\Omega}$ it follows that $$\Omega(W_i/W_{i-1},W_j/W_{j-1})=0,\ \text{ unless }\ i+j=k+1.$$
 Moreover, since every $W_j$ is $\varphi$-invariant and $\varphi$ is symmetric with respect to $\Omega$, it follows that, with respect to the above decomposition of $W$,
 $$\varphi\in H^0\bigg(\bigoplus_i\End(W_i/W_{i-1})\otimes K\bigg).$$ So, from Proposition \ref{prop:u*2n-alpha-stability}, $(W,\Omega,\varphi)$ is polystable.

The converse statement is
immediate by applying the stability condition of Proposition \ref{prop:u*2n-alpha-stability} the filtration
$0\subset W'\subset W'^{\perp_{\Omega}}\subset W$ if the $\varphi$-invariant subbundle $W'\subset W$ with $\deg(W')=0$ is isotropic, or to the filtration $0\subset W'^{\perp_{\Omega}}\subset W'\subset W$ if it is coisotropic.
\end{proof}

In order to construct moduli spaces, we need to consider $S$-equivalence classes of semistable $G$-Higgs bundles (cf. \cite{schmitt:2008}). For a stable $G$-Higgs bundle, its $S$-equivalence class coincides with its isomorphism class and for a strictly semistable $G$-Higgs bundle, its $S$-equivalence contains precisely one (up to isomorphism) representative which is polystable so this class can be thought as the isomorphism class of the unique polystable $G$-Higgs bundle which is $S$-equivalent to the given strictly semistable one.

These moduli spaces have been constructed by Schmitt in \cite{schmitt:2008}, using methods of Geometric Invariant Theory, showing that they carry a natural structure of complex algebraic variety.

\begin{definition}\label{moduli of Higgs bundles}
Let $X$ be a compact Riemann of genus $g\geq 2$. For a reductive Lie group $G$, the \emph{moduli space of $G$-Higgs bundles over $X$} is the complex analytic variety of isomorphism classes of polystable $G$-Higgs bundles. We denote it by $\M_G$:
$$\M_G=\{\text{Polystable }G\text{-Higgs bundles on }X\}/\sim.$$ For a fixed topological class $c$ of $G$-Higgs bundles, denote by $\M_G(c)$ the moduli space of $G$-Higgs bundles which belong to the class $c$.
\end{definition}

\begin{remark}
 If $G$ is an algebraic group then $\M_G$ has the structure of complex algebraic variety.
\end{remark}

\subsection{Deformation theory of $\U^*(2n)$-Higgs bundles}

In this section, we briefly recall the deformation theory of $G$-Higgs bundles and, in particular, the identification of the tangent space of $\M_G$ at the smooth points. All these basic notions can be found in detail in \cite{garcia-gothen-mundet:2008}.

The spaces $\liehc$ and $\liemc$ in the Cartan decomposition of $\liegc$ verify the relation $$[\liehc,\liemc]\subset\liemc$$ hence, given $v\in\liemc$, there is
an induced map $\ad(v)|_{\liehc}:\liehc\to\liemc$.
Applying this to a $G$-Higgs bundle $(E_{H^\CC},\varphi)$, we obtain the
following complex of sheaves on the curve $X$:
$$C_G^\bullet(E_{H^\CC},\varphi):\mathcal{O}(E_{H^\CC}(\liehc))\xrightarrow{\ad(\varphi)}\mathcal{O}(E_{H^\CC}(\liemc)\otimes K).$$

\begin{proposition}\label{deformation for Higgs}
Let $(E_{H^\CC},\varphi)$ be a $G$-Higgs bundle over $X$.
\begin{enumerate}
    \item The
infinitesimal deformation space of $(E_{H^\CC},\varphi)$ is isomorphic to the
first hypercohomology group $\HH^1(C_G^\bullet(E_{H^\CC},\varphi))$ of the
complex of sheaves $C_G^\bullet(E_{H^\CC},\varphi)$;
    \item There is a long exact sequence
\begin{equation*}
\begin{split}
0&\longrightarrow\HH^0(C_G^\bullet(E_{H^\CC},\varphi))\longrightarrow
H^0(E_{H^\CC}(\liehc))\longrightarrow
H^0(E_{H^\CC}(\liemc)\otimes
K)\longrightarrow\\
&\longrightarrow\HH^1(C_G^\bullet(E_{H^\CC},\varphi))\longrightarrow H^1(E_{H^\CC}(\liehc))\longrightarrow
H^1(E_{H^\CC}(\liemc)\otimes K)\longrightarrow\\
&\longrightarrow\HH^2(C_G^\bullet(E_{H^\CC},\varphi))\longrightarrow 0
\end{split}
\end{equation*}
where the maps $H^i(E_{H^\CC}(\liehc))\to H^i(E_{H^\CC}(\liemc)\otimes K)$ are induced by
$\mathrm{ad}(\varphi)$.
\end{enumerate}
\end{proposition}

Now, given a $\U^*(2n)$-Higgs bundle $(W,\Omega,\varphi)$, the complex $C_G^\bullet(E_{H^\CC},\varphi)$ defined above, becomes the complex of sheaves
$$C^\bullet(W,\Omega,\varphi):\Lambda_\Omega^2W\xrightarrow{\ad(\varphi)}S_\Omega^2W\otimes K$$
where $\Lambda_\Omega^2W$ denotes the bundle of endomorphisms of $W$ which are skew-symmetric with respect to $\Omega$, and where $\ad(\varphi)=[\varphi,-]$ is given by the Lie bracket.

Proposition \ref{deformation for Higgs} applied to the case of $\U^*(2n)$-Higgs bundles, yields the following.
\begin{proposition}\label{deformation complex for quadruples}
Let $(W,\Omega,\varphi)$ be a $\U^*(2n)$-Higgs bundle over $X$.
\begin{enumerate}
    \item The
infinitesimal deformation space of $(W,\Omega,\varphi)$ is isomorphic to
the first hypercohomology group
$\HH^1(C^\bullet(W,\Omega,\varphi))$ of
$C^\bullet(W,\Omega,\varphi)$. In particular, if $(W,\Omega,\varphi)$
represents a smooth point of $\M_d$, then
    $$T_{(W,\Omega,\varphi)}\M\simeq\HH^1(C^\bullet(W,\Omega,\varphi)).$$
    \item There is an exact sequence
\begin{equation*}
\begin{split}
0&\longrightarrow\HH^0(C^\bullet(W,\Omega,\varphi))\longrightarrow
H^0(\Lambda_\Omega^2W)\longrightarrow H^0(S_\Omega^2W\otimes
K)\longrightarrow\\
&\longrightarrow\HH^1(C^\bullet(W,\Omega,\varphi))\longrightarrow H^1(\Lambda_\Omega^2W)\longrightarrow
H^1(S_\Omega^2W\otimes K)\longrightarrow\\
&\longrightarrow\HH^2(C^\bullet(W,\Omega,\varphi))\longrightarrow 0
\end{split}
\end{equation*}
where
the maps $H^i(\Lambda_\Omega^2W)\to
H^i(S_\Omega^2W\otimes K)$ are induced by $\ad(\varphi)=[\varphi,-]$.
\end{enumerate}
\end{proposition}

The definition of simple $G$-Higgs bundle is given in \cite{garcia-gothen-mundet:2008} 
as follows.
\begin{definition}\label{general definition of simple}
A $G$-Higgs bundle $(E_{H^\CC},\varphi)$ is \emph{simple} if $\Aut(E_{H^\CC},\varphi)=\ker(\iota)\cap Z(H^\CC)$, where $Z(H^\CC)$ is the center of $H^\CC$ and $\iota$ is the isotropy representation (\ref{complex isotropy representation}).
\end{definition}

Contrary to the case of vector bundles, stability of a $G$-Higgs bundle does not imply that it is simple.

From Proposition \ref{deformation for Higgs}, one has that 
\begin{equation}\label{eq:dimH1}
\begin{split}
\dim\HH^1(C_G^\bullet(E_{H^\CC},\varphi))&=\chi(E_{H^\CC}(\liemc)\otimes K)-\chi(E_{H^\CC}(\liehc))+\\
&+\dim\HH^0(C_G^\bullet(E_{H^\CC},\varphi))+\dim\HH^2(C_G^\bullet(E_{H^\CC},\varphi))
\end{split}
\end{equation}
where $\chi=\dim H^0-\dim H^1$ denotes the Euler characteristic. The summand $$\chi(E_{H^\CC}(\liemc)\otimes K)+\chi(E_{H^\CC}(\liehc))$$ only depends on the topological class $c$ of $E_{H^\CC}$, which is fixed 
when we consider  $\M_G(c)$.
In order for a polystable $G$-Higgs bundle $(E_{H^\CC},\varphi)$ represent a smooth point of the moduli space $\M_G$, $\dim\HH^0(C_G^\bullet(E_{H^\CC},\varphi))$ and $\dim\HH^2(C_G^\bullet(E_{H^\CC},\varphi))$ must have the minimum possible value. Indeed, we have the following Proposition \ref{prop:stable, simple and H2=0 imply smooth} (cf. \cite{garcia-gothen-mundet:2008}), which gives sufficient conditions for a $G$-Higgs bundle $(E_{H^\CC},\varphi)$ represent a smooth point of $\M_G$. It uses the construction of a $G^\CC$-Higgs bundle from a $G$-Higgs bundle, which we now briefly explain.

Suppose that $G$ is a real form of $G^\CC$. The adjoint representation $$\Ad_{G^\CC}:G^\CC\to\Aut(\liegc)$$ of $G^\CC$ on its Lie algebra restricts to $H^\CC\subset G^\CC$ and the restriction splits as sum 
\begin{equation}\label{eq:split of adjoint rep of GC}
\Ad_{G^\CC}|_{H^\CC}=\Ad_{H^\CC}\oplus\iota
\end{equation}
where $\Ad_{H^\CC}:H^\CC\to\Aut(\liehc)$ is the adjoint representation of $H^\CC$ on $\liehc$ and $\iota:H^\CC\to\Aut(\liemc)$ is the isotropy representation (\ref{complex isotropy representation}). From a $G$-Higgs bundle $(E_{H^\CC},\varphi)$, we obtain a $G^\CC$-Higgs bundle as follows. Take $E_{G^\CC}$ to be the holomorphic $G^\CC$-principal bundle obtained from $E_{H^\CC}$ by extending the structure group through the inclusion $H^\CC\hookrightarrow G^\CC$. From this construction of $E_{G^\CC}$ and from (\ref{eq:split of adjoint rep of GC}), we have the spliting $$E_{G^\CC}\times_{G^\CC}\liegc=E_{H^\CC}\times_{H^\CC}\liegc=E_{H^\CC}\times_{H^\CC}\liehc\oplus E_{H^\CC}\times_{H^\CC}\liemc.$$
So, define $\varphi'\in H^0(E_{G^\CC}\times_{G^\CC}\liegc\otimes K)$ by considering the above spliting, taking $\varphi\in H^0(E_{H^\CC}\times_{H^\CC}\liemc\otimes K)$ and taking the zero section of $E_{H^\CC}\times_{H^\CC}\liehc$.
We say that \emph{$(E_{G^\CC},\varphi')$ is the $G^\CC$-Higgs bundle associated to the $G$-Higgs bundle $(E_{H^\CC},\varphi)$}.
Also, when we say that we view the $G$-Higgs bundle $(E_{H^\CC},\varphi)$ \emph{as a $G^\CC$-Higgs bundle}, it is this construction that we are referring to (see also \cite{bradlow-garcia-prada-gothen:2009}).

Now we can state the result.
\begin{proposition}\label{prop:stable, simple and H2=0 imply smooth}
Let $(E_{H^\CC},\varphi)$ be a polystable $G$-Higgs bundle which is stable, simple and such that $\HH^2(C_G^\bullet(E_{H^\CC},\varphi))=0$. Then it corresponds to a smooth point of the moduli space $\M_G(c)$. In particular, if $(E_{H^\CC},\varphi)$ is a simple $G$-Higgs bundle which is stable as a $G^{\CC}$-Higgs bundle, then it is a smooth point in the moduli space.
\end{proposition}

Let $(E_{H^\CC},\varphi)$ represent a smooth point of $\M_G(c)$. The \emph{expected dimension} of $\M_G(c)$ is given by  \begin{equation}\label{eq:expected dimension}
\chi(E_{H^\CC}(\liemc)\otimes K)-\chi(E_{H^\CC}(\liehc))+\dim\Aut(E_{H^\CC},\varphi).
\end{equation}
The actual dimension of the moduli space (if non-empty) can be strictly smaller than the expected dimension. This phenomenon occurs for example in $\M_{\U(p,q)}$, as explained in \cite{bradlow-garcia-prada-gothen:2003}, where there is a component of dimension strictly smaller than the expected one. In fact, in that component there are no stable objects.

\subsection{Stable and non-simple $\U^*(2n)$-Higgs bundles}

Our goal in this section is to give an explicit description of $\U^*(2n)$-Higgs bundles which are stable but not simple.

As an example of the above construction of a $G^\CC$-Higgs bundle associated to a $G$-Higgs bundle, and which will be important below, consider a $\U^*(2n)$-Higgs bundle $(W,\Omega,\varphi)$. Then, the corresponding $\GL(2n,\CC)$-Higgs bundle is simply $(W,\varphi)$. So we forget the symplectic form on the vector bundle $W$.

\begin{proposition}\label{prop:equivalence of semistability}
 Let $(W,\Omega,\varphi)$ be a $\U^*(2n)$-Higgs bundle and $(W,\varphi)$ be the corresponding $\GL(2n,\CC)$-Higgs bundle. Then $(W,\Omega,\varphi)$ is semistable if and only if $(W,\varphi)$ is semistable.
\end{proposition}
\proof
If $(W,\varphi)$ is semistable, then it is obvious, taking into account Propositions \ref{(poly)stability for usual Higgs bundles} and \ref{prop:orthogonal-stability}, that $(W,\Omega,\varphi)$ is semistable.

Suppose then that $(W,\Omega,\varphi)$ is semistable. Let $W'\subset W$ be a $\varphi$-invariant subbundle of $W$. Since $W'^{\perp_\Omega}$ is the subbundle of $W$ defined as the kernel of the projection $W\to W'^*$ given by $v\mapsto\Omega(v,-)$, and since $\deg(W)=0$, we have 
\begin{equation}\label{degperp=deg}
\deg(W'^{\perp_\Omega})=\deg(W').
\end{equation} The fact that $\varphi$ is symmetric with respect to $\Omega$, i.e. (\ref{symmetric higgs field}) holds, implies that $W'^{\perp_\Omega}$ is also $\varphi$-invariant.

Consider the exact sequence
\begin{equation}\label{direct sum sequence 0}
0\longrightarrow N\longrightarrow W'\oplus W'^{\perp_\Omega}\longrightarrow M\longrightarrow0, 
\end{equation}
where $N$ and $M$ are the saturations of the sheaves $W'\cap W'^{\perp_\Omega}$ and $W'+W'^{\perp_\Omega}$ respectively. We have that $M=N^{\perp_\Omega}$, so $$0\longrightarrow M\longrightarrow W\longrightarrow N^*\longrightarrow0$$
and thus $\deg(M)=\deg(N)$. It follows from (\ref{degperp=deg}) and (\ref{direct sum sequence 0}) that $\deg(W')=\deg(N)$. But, $N$ is clearly $\varphi$-invariant and also isotropic, so from the semistability of $(W,\Omega,\varphi)$, $\deg(N)\leq 0$ i.e. $\deg(W')\leq 0$. Hence, from Proposition \ref{(poly)stability for usual Higgs bundles}, $(W,\varphi)$ is semistable.
\endproof

\begin{proposition}\label{stable U*(2n)-Higgs bundles}
 Let $(W,\Omega,\varphi)$ be a $\U^*(2n)$-Higgs bundle. Then $(W,\Omega,\varphi)$ is stable if and only if $$(W,\Omega,\varphi)=\bigoplus(W_i,\Omega_i,\varphi_i)$$ where $(W_i,\Omega_i,\varphi_i)$ are $\U^*(\rk(W_i))$-Higgs bundles such that the $\GL(\rk(W_i),\CC)$-Higgs bundles $(W_i,\varphi_i)$ are stable and nonisomorphic.
\end{proposition}
\proof
Suppose that $(W,\Omega,\varphi)$ is stable. From Proposition \ref{prop:equivalence of semistability}, follows that $(W,\varphi)$ is semistable. If it is stable, then there is nothing to prove. So, assume that $(W,\varphi)$ is strictly semistable, and let $W'\subset W$ be a $\varphi$-invariant subbundle of $W$ of degree $0$. The stability of $(W,\Omega,\varphi)$ says that $W'$ is not isotropic.
As in the proof of the previous proposition, consider the exact sequence
\begin{equation}\label{direct sum sequence}
0\longrightarrow N\longrightarrow W'\oplus W'^{\perp_\Omega}\longrightarrow M\longrightarrow0,
\end{equation}
where $N$ and $M$ are the saturations of the sheaves $W'\cap W'^{\perp_\Omega}$ and $W'+W'^{\perp_\Omega}$ respectively. From the sequence
$$0\longrightarrow W'^{\perp_\Omega}\longrightarrow W\longrightarrow W'^*\longrightarrow0$$ we have $\deg(W'^{\perp_\Omega})=\deg(W')=0$ so, from (\ref{direct sum sequence}), $$\deg(N)+\deg(M)=0.$$
Recall again that, since $W'$ is $\varphi$-invariant, then $W'^{\perp_\Omega}$ is also $\varphi$-invariant, so $N\subset W$ is $\varphi$-invariant as well and, since it is isotropic, we must have $\deg(N)<0$, if $N\neq 0$. But, if this occurs, we have $\deg(M)>0$ contradicting (since $M$ is $\varphi$-invariant) the semistability of $(W,\varphi)$. We must therefore have $N=0$, hence $$(W,\varphi)=(W',\varphi|_{W'})\oplus (W'^{\perp_\Omega},\varphi|_{W'^{\perp_\Omega}}).$$
Now, $W'\ncong W'^{\perp_\Omega}$. In fact, if $W'\cong W'^{\perp_\Omega}$ then the inclusion $W'\subset W'\oplus W'=W$ given by $w\mapsto(w,\sqrt{-1}w)$ gives rise to a degree $0$ isotropic, $\varphi$-invariant subbundle of $W$, contradicting the stability of $(W,\Omega,\varphi)$. Finally, notice that we must have
$$\omega=\begin{pmatrix}
     \omega_1 & 0 \\
     0 & \omega_2
\end{pmatrix}$$ with respect to the decomposition $W=W'\oplus W'^{\perp_\Omega}$, where $\omega_1:W'\to W'^*$ and $\omega_2:W'^{\perp_\Omega}\to(W'^{\perp_\Omega})^*$ are skew-symmetric isomorphisms. The symplectic form $\Omega$ therefore splits into a sum of symplectic forms $\Omega_1\oplus\Omega_2$, and we have a splitting 
$$(W,\Omega,\varphi)=(W',\Omega_1,\varphi_1)\oplus(W'^{\perp_\Omega},\Omega_2,\varphi_2).$$

Now, if $(W',\varphi_1)$ is stable as a $\GL(\rk(W'),\CC)$-Higgs bundle and if the same happens to $(W'^{\perp_\Omega},\varphi_2)$, then we are done. If not, then we repeat the argument and, by induction on the rank of $W$, we see that $(W,\Omega,\varphi)$ has the desired form.

To prove the converse, suppose that $$(W,\Omega,\varphi)=\bigoplus(W_i,\Omega_i,\varphi_i)$$ as stated and let $W'\subset W$ be a $\varphi$-invariant subbundle of degree $0$. Since each $(W_i,\varphi_i)$ is a stable $\GL(\rk(W_i),\CC)$-Higgs bundle, then the projection $W'\to W_i$ must be either zero or surjective. Thus, $(W',\varphi|_{W'})$ is a direct sum of some of the $(W_i,\varphi_i)$, so $W'$ is not isotropic and therefore $(W,\Omega,\varphi)$ is stable.
\endproof

Applying Definition \ref{general definition of simple} to the case $G=\U^*(2n)$, we have:
\begin{lemma}
A $\U^*(2n)$-Higgs bundle $(W,\Omega,\varphi)$ is simple if and only if $\Aut(W,\Omega,\varphi)=\Z/2$. 
\end{lemma}

\begin{corollary}\label{cor:stable and simple imply smooth as GL}
 Let $(W,\Omega,\varphi)$ be a stable $\U^*(2n)$-Higgs bundle. Then $(W,\Omega,\varphi)$ is simple if and only if the $\GL(2n,\CC)$-Higgs bundle $(W,\varphi)$ is stable.
\end{corollary}
\proof
Since $(W,\Omega,\varphi)$ is stable, we have, from Proposition \ref{stable U*(2n)-Higgs bundles},
\begin{equation}\label{eq:decomposition of U*-Higgs}
(W,\Omega,\varphi)=\bigoplus_{i=1}^r(W_i,\Omega_i,\varphi_i),
\end{equation}
so $$(W,\varphi)=\bigoplus_{i=1}^r(W_i,\varphi_i)$$ where $(W_i,\varphi_i)$ are stable Higgs bundles. Since stable Higgs bundles are simple (cf. \cite{hitchin:1987}) then $\Aut(W_i,\varphi_i)=\CC^*$. This means that, for each $i$, $\Aut(W_i,\Omega_i,\varphi_i)=\Z/2$, because the automorphisms must preserve the symplectic form $\Omega_i$. From (\ref{eq:decomposition of U*-Higgs}), we have therefore $$\Aut(W,\Omega,\varphi)=(\Z/2)^r.$$ It follows that $(W,\Omega,\varphi)$ is simple if and only if $r=1$ i.e. $(W,\varphi)$ is a stable Higgs bundle.
\endproof

Now the description of stable and non-simple $\U^*(2n)$-Higgs bundles is immediately obtained.

\begin{proposition}\label{prop:description of stable and non-simple}
 A $\U^*(2n)$-Higgs bundle is stable and non-simple if and only if it decomposes as a direct sum of stable and simple $\U^*(2n)$-Higgs bundles. In other words, a $\U^*(2n)$-Higgs bundle $(W,\Omega,\varphi)$ is stable and non-simple if and only if $$(W,\Omega,\varphi)=\bigoplus_{i=1}^r(W_i,\Omega_i,\varphi_i)$$ where $(W_i,\Omega_i,\varphi_i)$ are stable and simple $\U^*(\rk(W_i))$-Higgs bundles and $r>1$.
\end{proposition}

The following result will be important below. It is straightforward from Proposition \ref{prop:stable, simple and H2=0 imply smooth}, from the fact that the complexification of $\U^*(2n)$ is $\GL(2n,\CC)$ and from Corollary \ref{cor:stable and simple imply smooth as GL}:

Let $\M_{\U^*(2n)}$ denote the moduli space of polystable $\U^*(2n)$-Higgs bundles.

\begin{proposition}\label{prop:stable and simple imply smooth}
A stable and simple $\U^*(2n)$-Higgs bundle corresponds to a smooth point of the moduli space $\M_{\U^*(2n)}$.
\end{proposition}

So, from \cite{schmitt:2008}, at a point of $\M_{\U^*(2n)}$ represented by a stable and simple object, there exists a local universal family, hence the dimension of the component of $\M_{\U^*(2n)}$ containing that point is the expected dimension given by (\ref{eq:expected dimension}), which, for $G=\U^*(2n)$ is easily seen to be equal to $$4n^2(g-1).$$

\subsection{Polystable $\U^*(2n)$-Higgs bundles}

Now we look at polystable $\U^*(2n)$-Higgs bundles. First notice that we can realize $\GL(n,\CC)$ as a subgroup of $\U^*(2n)$, using the injection $$A\mapsto\begin{pmatrix}
A & 0\\
0 & \overline A\end{pmatrix}.$$ When restricted to the unitary group $\U(n)\subset\GL(n,\CC)$ we obtain the injection
$$A\mapsto\begin{pmatrix}
A & 0\\
0 & (A^t)^{-1}\end{pmatrix}.$$

\begin{theorem}\label{thm:description of polystable}
 Let $(W,\Omega,\varphi)$ be a polystable $\U^*(2n)$-Higgs bundle. There is a decomposition of $(W,\Omega,\varphi)$ as a sum of stable $G_i$-Higgs bundles, where $G_i$ is one of the following subgroups of $\U^*(2n)$: $\U^*(2n_i)$, $\GL(n_i,\CC)$, $\Sp(2n_i)$ or $\U(n_i)$ ($n_i\leq n$).
\end{theorem}
\proof
Since $(W,\Omega,\varphi)$ is polystable, we know, from Proposition \ref{prop:u*2n-alpha-stability}, that for every filtration $$\mathcal W=(0=W_0\subsetneq W_1\subsetneq W_2\subsetneq\dots\subsetneq W_k=W)$$ such that $W_{k-j}=W_j^{\perp_\Omega}$, and any
 $$\lambda\in\{(\lambda_1,\lambda_2,\dots,\lambda_k)\in\R^k
\mid \lambda_i<\lambda_{i+1}\text{ and
}\lambda_{k-i+1}=-\lambda_i\text{ for any }i\},$$ such that $\varphi\in H^0(N(\mathcal W,\lambda))$ and
 $d(\mathcal W,\lambda)=0$, there is an isomorphism
 \begin{equation}\label{eq:polystable decomposition of W}
W\simeq W_1\oplus W_2/W_1\oplus\dots\oplus W_k/W_{k-1}  
 \end{equation}
 such that
\begin{equation}\label{eq:orthogonal decomposition of Omega}
\Omega(W_i/W_{i-1},W_j/W_{j-1})=0,\ \text{ unless }\ i+j=k+1
\end{equation}
 and that, via this isomorphism,
\begin{equation}\label{eq:orthogonal decomposition of Higgs field}
\varphi\in H^0\bigg(\bigoplus_i\End(W_i/W_{i-1})\otimes K\bigg).
\end{equation}

Now we analyze the possible cases. Conditions (\ref{eq:orthogonal decomposition of Omega}) and (\ref{eq:orthogonal decomposition of Higgs field}) tell us that, with respect to decomposition (\ref{eq:polystable decomposition of W}), we have
\begin{equation}\label{eq:matrix representation of symplectic form}
 \omega=\begin{pmatrix}
    0 & 0 & 0 & \dots & -\omega_1^t \\
    \vdots &  \dots & 0 & \dots & \vdots \\
    0 & 0 &  \dots &  \dots & 0\\
    0 & \omega_2 & 0 & \dots & 0 \\
    \omega_1 & 0 & 0 & \dots & 0\
  \end{pmatrix},
\end{equation} where $\omega_i:W_i/W_{i-1}\stackrel{\cong}{\longrightarrow}(W_{k+1-i}/W_{k-i})^*$ is the isomorphism induced by $\Omega$, and that $$\varphi(W_i/W_{i-1})\subset W_i/W_{i-1}\otimes K,$$ for all $i=1,\ldots,k$, so we write $$\varphi_i=\varphi|_{W_i/W_{i-1}}.$$

Hence, if $i\neq \frac{k+1}{2}$, from (\ref{eq:orthogonal decomposition of Omega}), the symplectic form $\Omega$ does not restricts to a symplectic form on $W_i/W_{i-1}$, and we deduce that $$(W_i/W_{i-1},\varphi_i)$$ is a $\GL(\rk(W_i/W_{i-1}),\CC)$-Higgs bundle, being a $\U(\rk(W_i/W_{i-1}))$-Higgs bundle if and only if $\varphi_i=0$.

On the other hand, the symplectic form $\Omega$ restricts to a symplectic form $\Omega_{\frac{k+1}{2}}$ on $W_{\frac{k+1}{2}}/W_{\frac{k-1}{2}}$, and we deduce that $$(W_\frac{k+1}{2}/W_{\frac{k-1}{2}},\Omega_{\frac{k+1}{2}},\varphi_{\frac{k+1}{2}})$$ is a $\U^*(\rk(W_\frac{k+1}{2}/W_{\frac{k-1}{2}}))$-Higgs bundle, being a $\Sp(\rk(W_\frac{k+1}{2}/W_{\frac{k-1}{2}}))$-Higgs bundle if and only if $\varphi_{\frac{k+1}{2}}=0$. Of course, this case can only occur if $k$ is odd.

Each summand in this decomposition is also polystable (one way of seeing this is by using the Hitchin-Kobayashi correspondence between polystable $G$-Higgs bundles and solutions to the Hitchin equations; cf. \cite{garcia-gothen-mundet:2008}). Hence, for each summand which is a $\Sp(2n_i)$- or $\GL(n_i,\CC)$- or $\U(n_i)$-Higgs bundle we know that we can continue the process for these groups, until we obtain a decomposition where all summands are stable $\Sp(2n_i)$- or $\GL(n_i,\CC)$- or $\U(n_i)$-Higgs bundles: for $\U(n_i)$-Higgs bundles (i.e. holomorphic vector bundles) this is proved in \cite{seshadri:1967}; the proof for the case of $\GL(n,\CC)$-Higgs bundles can be found in \cite{nitsure:1991} and for $\Sp(2n)$-Higgs bundles (i.e. symplectic vector bundles) this is proved in \cite{hitching:2005} (see also \cite{ramanan:1981}).
On the other hand, for $\U^*(2n_i)$-Higgs bundle we simply iterate the above process. Finally we obtain a decomposition where all summands are stable $G_i$-Higgs bundles.
\endproof

\section{The Hitchin proper functional and the minima subvarieties}\label{morse quadruples}

Here we use the method introduced by Hitchin in \cite{hitchin:1987} to
study the topology of moduli space $\M_G$ of $G$-Higgs bundles.

Define $$f:\M_G(c)\longrightarrow\R$$ by
\begin{equation}\label{proper function}
f(E_{H^\CC},\varphi)=\|\varphi\|_{L^2}^2=\int_{X}|\varphi|^2\mathrm{dvol}.
\end{equation}
This function $f$ is usually called the \emph{Hitchin functional}.

Here we are using the \emph{harmonic metric} (cf. \cite{corlette:1988,donaldson:1987}) on $E_{H^\CC}$ to define
$\|\varphi\|_{L^2}$. So we are using the identification between $\M_G(c)$ with the space of gauge-equivalent solutions
to Hitchin's equations. We opt to work with $\M_G(c)$, because in this case we have more algebraic tools at our disposal. We shall make use of the tangent space of $\M_G(c)$, and we know from \cite{hitchin:1987} that the above identification induces a diffeomorphism between the corresponding tangent spaces.

Hitchin proved in \cite{hitchin:1987, hitchin:1992} that the function $f$ is proper and therefore it attains a minimum on each closed subspace of $\M_G=\bigcup_c\M_G(c)$. Moreover, we have the following result from general topology.
\begin{proposition}\label{proper}
Let $\M'\subseteq\M_G$ be a closed subspace and let $\N'\subset\M'$ be the subspace of local minima of $f$ on $\M'$. If $\N'$ is connected then so is $\M'$.
\end{proposition}

In our case, the Hitchin functional $$f:\M_{\U^*(2n)}\longrightarrow\R$$
is given by
\begin{equation}\label{Hitchin proper function quadruples}
f(W,\Omega,\varphi)=\|\varphi\|_{L^2}^2=\frac{\sqrt{-1}}{2}\int_{X}\mathrm{tr}(\varphi\wedge\varphi^*)\mathrm{dvol}.
\end{equation}


Recall from Proposition \ref{prop:stable and simple imply smooth} which guarantees that a stable and simple $\U^*(2n)$-Higgs bundle represents a smooth point on $\M_{\U^*(2n)}$.

Away from the singular locus of
$\M_{\U^*(2n)}$, the Hitchin functional $f$ is a moment map for the Hamiltonian
$S^1$-action on $\M_{\U^*(2n)}$ given by
\begin{equation}\label{circle action Higgs}
(V,\varphi)\mapsto(V,e^{\sqrt{-1}\theta}\varphi).
\end{equation}
From this it follows immediately that a smooth point of $\M_{\U^*(2n)}$ is a
critical point of $f$ if and only if is a fixed point of the
$S^1$-action. Let us then study the fixed point set of the given
action (this is analogous to \cite{hitchin:1992} and \cite{bradlow-garcia-prada-gothen:2004}).

Let $(W,\Omega,\varphi)$ represent a stable and simple (hence smooth) fixed point. Then either $\varphi=0$
or (since the action is on $\M_{\U^*(2n)}$) there is a
one-parameter family of gauge transformations
$g(\theta)$ such that
$g(\theta)\cdot(W,\Omega,\varphi)=(W,\Omega,e^{\sqrt{-1}\theta}\varphi)$.

In the latter case, let
\begin{equation}\label{eq:infinitesimal gauge}
\psi=\frac{d}{d\theta}g(\theta)|_{\theta=0}
\end{equation}
be the infinitesimal gauge
transformation generating this family. $(W,\Omega,\varphi)$ is then what is called a \emph{complex variation of
Hodge structure} or a \emph{Hodge bundle} (cf. \cite{hitchin:1987,hitchin:1992,simpson:1992}). This means that
$$(W,\varphi)=\Big(\bigoplus F_j,\sum\varphi_j\Big)$$ where the $F_j$'s are the eigenbundles of
the infinitesimal gauge transformation $\psi$: over $F_j$,
\begin{equation}\label{psi over Fj}
\psi=\sqrt{-1}j\in\CC,
\end{equation}
and where $\varphi_j=\varphi|_{F_j}$ is a map
\begin{equation}\label{eq:Hodge decomposition of varphi}
\varphi_j:F_j\longrightarrow F_{j+1}\otimes K.
\end{equation}

Since $g(\theta)$ is an automorphism of $(W,\Omega)$, it follows from (\ref{eq:infinitesimal gauge}) that $\psi$ is
skew-symmetric with respect to $\Omega$. Thus, using (\ref{psi over Fj}) we have that, if $v_j\in F_j$ and
$v_i\in F_i$,
$$\sqrt{-1}j\Omega(v_j,v_i)=\Omega(\psi v_j,v_i)=-\Omega(v_j,\psi v_i)=-\sqrt{-1}l\Omega(v_j,v_i).$$
Then $F_j$ and $F_i$ are therefore orthogonal under $\Omega$ unless $i+j=0$, and therefore $\omega:W\to W^*$ yields an isomorphism
\begin{equation}\label{FjcongF-j*L}
\omega_j=\omega|_{F_j}:F_j\stackrel{\cong}{\longrightarrow}F_{-j}^*.
\end{equation}
This means that
\begin{equation}\label{eq:decomposition of W}
W=F_{-m}\oplus\dots\oplus F_m
\end{equation} for some $m\geq 1/2$ integer or half-integer.

Using these isomorphisms and (\ref{symmetric higgs field}), we see that $$(\varphi_{-j-1}^t\otimes
1_K)\omega_j=(\omega_{j+1}\otimes 1_K)\varphi_j$$ for $j\in\{-m,\dots,m\}$.

The Cartan decomposition of $\liegc$ induces a decomposition of vector bundles
$$E_{H^\CC}(\liegc)=E_{H^\CC}(\liehc)\oplus E_{H^\CC}(\liemc)$$
where $E_{H^\CC}(\liegc)$ (resp. $E_{H^\CC}(\liehc)$) is the adjoint bundle, associated to the adjoint representation of $H^\CC$ on $\liegc$ (resp. $\liehc$).
For the group $\U^*(2n)$, we have
$E_{H^\CC}(\liegc)=\End(W)$ and
we already know that
$E_{H^\CC}(\liehc)=\Lambda_\Omega^2W$ and
$E_{H^\CC}(\liemc)=S_\Omega^2W$.
The involution in $\End(W)$ defining the above decomposition
is $\theta:\End(W)\to\End(W)$ defined by
\begin{equation}\label{involution theta}
\theta(A)=-(\omega A\omega^{-1})^t.
\end{equation}
Its $+1$-eigenbundle is $\Lambda_\Omega^2W$ and its
$-1$-eigenbundle is $S_\Omega^2W$.

We also have a decomposition of this vector bundle as
\begin{equation}\label{eq:decomposition of End}
\End(W)=\bigoplus_{k=-2m}^{2m}U_k
\end{equation}
where
$$U_k=\bigoplus_{i-j=k}\Hom(F_j,F_i).$$ From (\ref{psi over Fj}), this is the $\sqrt{-1}k$-eigenbundle for the adjoint action
$\ad(\psi):\End(W)\to\End(W)$
of $\psi$. We say that $U_k$ is the subspace of
$\End(W)$ with \emph{weight} $k$.

Write
$$U_{i,j}=\Hom(F_j,F_i).$$ The restriction of the involution
$\theta$, defined in (\ref{involution theta}), to $U_{i,j}$ gives an isomorphism
\begin{equation}\label{theta sends Uij to U-i,-j}
\theta:U_{i,j}\stackrel{\cong}{\longrightarrow} U_{-j,-i}
\end{equation} so $\theta$ restricts to
$$\theta:U_k\longrightarrow U_k.$$

Write $$U^+=\Lambda_\Omega^2W\ \text{ and }\ U^-=S_\Omega^2W$$ so that
$E_{H^\CC}(\liehc)=U^+$ and
$E_{H^\CC}(\liemc)=U^-$. Let also $$U_k^+=U_k\cap U^+$$ and
$$U_k^-=U_k\cap U^-$$ so that $U_k=U_k^+\oplus U_k^-$ is the
corresponding eigenbundle decomposition. Hence $$U^+=\bigoplus_k
U_k^+$$ and $$U^-=\bigoplus_k U_k^-.$$ Observe that $\varphi\in
H^0(U_1^-\otimes K)$.

The map $\ad(\varphi)=[\varphi,-]$ interchanges $U^+$ with $U^-$ and
therefore maps $U_k^\pm$ to $U_{k+1}^\mp\otimes K$. So, for each
$k$, we have a weight $k$ subcomplex of the complex $C^\bullet(W,\Omega,\varphi)$ defined in Proposition \ref{deformation complex for quadruples}:
\begin{equation}\label{eq:subcomplex}
C^\bullet_k(W,\Omega,\varphi):U_k^+\xrightarrow{\ad(\varphi)}U_{k+1}^-\otimes K.
\end{equation}

From Propositions \ref{deformation complex for quadruples} and \ref{prop:stable and simple imply smooth}, if a $\U^*(2n)$-Higgs bundle $(W,\Omega,\varphi)$ is stable and simple, its infinitesimal deformation
space is
$$\HH^1(C^\bullet(W,\Omega,\varphi))=\bigoplus_k\HH^1(C^\bullet_k(W,\Omega,\varphi)).$$
We say that $\HH^1(C^\bullet_k(W,\Omega,\varphi))$ is the subspace of
$\HH^1(C^\bullet(W,\Omega,\varphi))$ with \emph{weight} $k$.

By Hitchin's computations in \cite{hitchin:1992} (see also \cite{garcia-prada-gothen-munoz:2007}), we
have the following result which gives us a way to compute the
eigenvalues of the Hessian of the Hitchin functional $f$ at a smooth critical point.

\begin{proposition}\label{eigen}
Let $(W,\Omega,\varphi)$ be a smooth $\U^*(2n)$-Higgs bundle which represents a critical
point of the Hitchin function $f$. The eigenspace of the Hessian of $f$
corresponding to the eigenvalue $k$ is
$$\HH^1(C^\bullet_{-k}(W,\Omega,\varphi)).$$ In particular,
$(W,\Omega,\varphi)$ is a local minimum of $f$ if and only if
$\HH^1(C^\bullet(W,\Omega,\varphi))$ has no subspaces with
positive weight.
\end{proposition}

For each $k$, consider the complex (\ref{eq:subcomplex}) and let $$\chi(C_k^\bullet(W,\Omega,\varphi))=\dim\mathbb{H}^0(C_k^\bullet(W,\Omega,\varphi))-\dim\mathbb{H}^1(C_k^\bullet(W,\Omega,\varphi))+\dim\mathbb{H}^2(C_k^\bullet(W,\Omega,\varphi)).$$

\begin{lemma}\label{lem:adphi isomo iff char=0}
Let $(W,\Omega,\varphi)$ be a stable $\U^*(2n)$-Higgs bundle which corresponds to a critical point of $f$. Then 
$\chi(C_k^\bullet(W,\Omega,\varphi))\leq (g-1)(2\rk(\ad(\varphi)|_{U_k^+})-\rk(U_k^+)-\rk(U_{k+1}^-))$. Furthermore,
$\chi(C_k^\bullet(W,\Omega,\varphi))=0$ if and only if $\mathrm{ad}(\varphi)|_{U_k^+}:U_k^+\to U_{k+1}^-\otimes K$ is an isomorphism.
\end{lemma}
\proof
This is essentially Lemma 3.11 of \cite{bradlow-garcia-prada-gothen:2008} (see also Proposition 4.4 of \cite{bradlow-garcia-prada-gothen:2003}). The proof in those papers is for $\GL(n,\CC)$ and $\U(p,q)$-Higgs bundles, but the argument works in the general setting of $G$-Higgs bundles (see Remark 4.16 of \cite{bradlow-garcia-prada-gothen:2003}): the key facts are that for a stable $G$-Higgs bundle, $(E_{H^\CC},\varphi)$, the Higgs vector bundle $(E_{H^\CC}\times_{\Ad}\liegc,\ad(\varphi))$ is semistable, and that there is a natural $\ad$-invariant isomorphism $E_{H^\CC}\times_{\Ad}\liegc\cong(E_{H^\CC}\times_{\Ad}\liegc)^*$ given by an invariant pairing on $\liegc$ (e.g. the Killing form). So we will only give a sketch of the proof here.

In the following we shall use the abbreviated notations $C_k^\bullet=C^\bullet_k(W,\Omega,\varphi)$ and
$$\varphi_k^{\pm}=\ad(\varphi)|_{U_k^{\pm}}: U_k^+ \longrightarrow U_{k+1}^-\otimes K.$$

By the Riemann-Roch theorem we have
\begin{equation}\label{eq:RR-C_k}
    \chi(C_k^\bullet) = (1-g)\bigl(\rk(U_k^+)+\rk(U_{k+1}^-)\bigr) +\deg(U_k^+) - \deg(U_{k+1}^-),
\end{equation}
thus we can prove the inequality stated in the lemma by estimating the difference $\deg(U_k^+) - \deg(U_k^-)$.  In order to
do this, we note first that there are short exact sequences of sheaves
$$0 \to \ker(\varphi_k^+) \to U_k^+ \to \im(\varphi_k^+) \to 0$$ and $$0 \to \im(\varphi_k^+) \to U_{k+1}^-\otimes K\to \coker(\varphi_k^+)\to 0.$$
It follows that
\begin{equation}\label{eq:degUU1}
\deg(U_k^+) - \deg(U_{k+1}^-)=\deg(\ker(\varphi_k^+)) + (2g-2)\rk(U_{k+1}^-) - \deg(\coker(\varphi_k^+)).
\end{equation}

The following inequalities are proved in the proof of Lemma 3.11 in \cite{bradlow-garcia-prada-gothen:2008}:
\begin{align}
 \deg(\ker(\varphi_k^+)) &\leq 0, \label{eq:3}\\
 -\deg(\coker(\varphi_k^+)) &\leq (2g-2)\bigl(-\rk(U_{k+1}^-) + \rk(\varphi_k^+)). \label{eq:4}
\end{align}

Combining (\ref{eq:3}) and (\ref{eq:4}) with \eqref{eq:degUU1} we obtain $$\deg(U_k^+)-\deg(U_{k+1}^-)\leq (2g-2)\rk(\varphi_k^+),$$ which, together with \eqref{eq:RR-C_k}, proves the inequality stated in the lemma.

Finally, if  $\chi(C_k^\bullet) = 0$ then $$\rk(\varphi_k^+) = \rk(U_k^+) =
  \rk(U_{k+1}^- \otimes K)$$ hence $\deg(\ker(\varphi_k^+))=0$. Moreover, it is shown again in the proof of Lemma 3.11 \cite{bradlow-garcia-prada-gothen:2008} that $\deg(\coker(\varphi_k^+))=0$. Thus, from (\ref{eq:degUU1}),
$$\deg(U_k^+)=\deg(U_{k+1}^-\otimes K),$$ showing that $\varphi_k^+$ is an isomorphism.
\endproof

The following result is fundamental for the description of the stable and simple local minima of $f$.

\begin{theorem}\label{ad}
Let $(W,\Omega,\varphi)\in\M_{\U^*(2n)}$ be a stable and simple critical point of the Hitchin functional $f$. Then
$(W,\Omega,\varphi)$ is a local minimum if and only if either $\varphi=0$ or $$\mathrm{ad}(\varphi)|_{U_k^+}:U_k^+\longrightarrow U_{k+1}^-\otimes K$$ is an isomorphism for all $k\geq 1$.
\end{theorem}
\proof
Suppose $\varphi\neq 0$ and that $\ad(\varphi)|_{U_k^+}$ is an isomorphism for every $k\geq 1$. Then, Lemma \ref{lem:adphi isomo iff char=0} says that this is equivalent to 
$$\dim\mathbb{H}^1(C_k^\bullet(W,\Omega,\varphi))=\dim\mathbb{H}^0(C_k^\bullet(W,\Omega,\varphi))+\dim\mathbb{H}^2(C_k^\bullet(W,\Omega,\varphi))$$ for all $k\geq 1$. Now, since $(W,\Omega,\varphi)$ is stable and simple, then it is stable as a $\GL(2n,\CC)$-Higgs bundle, by Corollary \ref{cor:stable and simple imply smooth as GL}. Furthermore, $\U^*(2n)$ is semisimple, so from Proposition 3.17 of \cite{garcia-gothen-mundet:2008} follows that $\mathbb{H}^0(C^\bullet(W,\Omega,\varphi))=\mathbb{H}^2(C^\bullet(W,\Omega,\varphi))=0$, so $\mathbb{H}^0(C_k^\bullet(W,\Omega,\varphi))=\mathbb{H}^2(C_k^\bullet(W,\Omega,\varphi))=0$ for every $k\geq 1$.
Then $\mathbb{H}^1(C_k^\bullet(W,\Omega,\varphi))=0$ for every $k\geq 1$ and the result follows from Proposition \ref{eigen}.

The converse statement is now immediate.
\endproof

Using this, one can describe the smooth local minima of the Hitchin functional $f$.

\begin{proposition}\label{min}
Let the $\U^*(2n)$-Higgs bundle $(W,\Omega,\varphi)$ be a critical point of the Hitchin functional
$f$ such that $(W,\Omega,\varphi)$ is stable and simple (hence smooth). Then $(W,\Omega,\varphi)$ represents a local minimum if and only if $\varphi=0$.
\end{proposition}
\proof Suppose that $(W,\Omega,\varphi)$ is a critical point of $f$ with $\varphi\neq 0$. Hence, as explained above, we have the decompositions (\ref{eq:decomposition of W}) and (\ref{eq:decomposition of End}) of $W$ and of $\End(W)$ respectively.

Consider $$\mathrm{ad}(\varphi)|_{U_{2m}^+}:U_{2m}^+\longrightarrow U_{2m+1}^-\otimes K.$$
We have that $U_{2m+1}^-\otimes K=0$, but $U_{2m}^+\neq 0$.
Indeed, if $U_{2m}^+=0$, then $$\Hom(F_{-m},F_m)=U_{2m}=U_{2m}^-$$ i.e. given any $g:F_{-m}\to F_m$, we would have $g\in S_\Omega^2W$, thus $$\omega_mg=g^t\omega_{-m}=-(\omega_mg)^t$$ where $\omega_{\pm m}$ are the isomorphisms defined in (\ref{FjcongF-j*L}). In other words, $\omega_mg\in H^0(\Lambda^2F_{-m}^*)$, for any $g$. But $\omega_m$ is an isomorphism, so any map $F_{-m}\to F_{-m}^*$ is of the form $\omega_mg$, for some $g$. This shows that 
$H^0(\Hom(F_{-m},F_{-m}^*))=H^0(\Lambda^2F_{-m}^*)$ which is clearly not possible.

So, $U_{2m}^+\neq 0$, therefore $\ad(\varphi)|_{U_{2m}^+}$ is not an isomorphism and by the previous theorem, $(W,\Omega,\varphi)$ is not a local minimum of the Hitchin functional.
\endproof


In \cite{hitchin:1992}, Hitchin observed that the Hitchin functional is additive with respect to direct sum of Higgs bundles. In our case this means that $f(\bigoplus(V_i,\Omega_i,\varphi_i))=\sum f(V_i,\Omega_i,\varphi_i)$.

\begin{proposition}\label{min for stable and not simple}
A stable $\U^*(2n)$-Higgs bundle $(W,\Omega,\varphi)$ represents a local minimum of $f$ if and
only if $\varphi=0$.
\end{proposition}
\proof If $(W,\Omega,\varphi)$ is simple, then this is true from Proposition \ref{min}.
So, assume that the local minimum $(W,\Omega,\varphi)$ of $f$ is stable and non-simple. Then, from Proposition \ref{prop:description of stable and non-simple}, we know that $(W,\Omega,\varphi)$ decomposes as a direct sum of stable and simple $\U^*(2n_i)$-Higgs bundles on the corresponding lower rank moduli spaces. Moreover, using the additivity of $f$, we know that these are also local minima of $f$. So, in those moduli spaces we can apply Proposition \ref{min}, and the additivity of $f$ implies that the result follows.
\endproof

Now we can give the description of the subvariety of local minima of the Hitchin functional $f$.

\begin{theorem}\label{min for polystable}
A polystable $\U^*(2n)$-Higgs bundle $(W,\Omega,\varphi)$ represents a local minimum if and
only if $\varphi=0$.
\end{theorem}
\proof
From Theorem \ref{thm:description of polystable} we know that a polystable minima of $f$ decomposes as a direct sum of stable $G_i$-Higgs bundles where $G_i=\U^*(2n_i),\,\Sp(2n_i),\,\GL(n_i,\CC)$ or $\U(n_i)$. Now, for the groups $\Sp(n_i)$ or $\U(n_i)$ it is clear that the local minima of $f$ on the corresponding lower rank moduli spaces must have zero Higgs field (these groups are compact). For $\GL(n_i,\CC)$ it is well-known (cf. \cite{hitchin:1987}) that stable local minima of $f$ on the corresponding lower rank moduli space must also have $\varphi_i=0$. For stable $\U^*(2n_i)$-Higgs bundle, we can apply Proposition \ref{min for stable and not simple} to draw the same conclusion, and the result is proved.
\endproof

\section{Connected components of the space of $\U^*(2n)$-Higgs bundles}\label{compM}

From Theorem \ref{min for polystable} we conclude that the subvariety $\N_{\U^*(2n)}$ of local minima of the Hitchin functional $f:\M_{\U^*(2n)}\to\R$ is the moduli space of $\Sp(2n,\CC)$-principal  bundles or, 
in the language of Higgs bundles, is the moduli space of $\Sp(2n)$-Higgs bundles: $$\N_{\U^*(2n)}\cong\M_{\Sp(2n)}.$$

Ramanathan has shown \cite{ramanathan:1996a,ramanathan:1996b} that 
if $G$ is a connected reductive group then there is a bijective correspondence 
between $\pi_0$ of the moduli space of $G$-principal bundles  and $\pi_1G$. Hence, 
since $\Sp(2n)$ is simply-connected, it follows that $\M_{\Sp(2n)}$ is connected and, therefore, 
the same is true for $\N_{\U^*(2n)}$. So, using Proposition \ref{proper}, we can state our result.

\begin{theorem}
 Let $X$ be a compact Riemann surface of genus $g\geq 2$ and let $\M_{\U^*(2n)}$ be the moduli space of $\U^*(2n)$-Higgs bundles. Then $\M_{\U^*(2n)}$ is connected.
\end{theorem}

\end{document}